\documentclass[12pt]{amsart}
\usepackage{amssymb,latexsym}

\usepackage{enumitem}
\usepackage{graphicx}
\usepackage[T1]{fontenc}
\newtheorem{theorem}{Theorem}[section]
\newtheorem{corollary}[theorem]{Corollary}
\newtheorem{lemma}[theorem]{Lemma}
\newtheorem{proposition}[theorem]{Proposition}

\theoremstyle{definition}
\newtheorem{definition}[theorem]{Definition}
\newtheorem{remark}[theorem]{Remark}
\newtheorem{example}[theorem]{Example}

\usepackage{color}

\makeatletter
\@namedef{subjclassname@2010}{%
	\textup{2010} Mathematics Subject Classification}
\makeatother

\begin{document}
	
	\baselineskip=17pt

	\title{ On Diagonal bimodules of \'{e}tale groupoid $C^*$-algebras}

	\author[R. Liu]{Rui Liu}
	\address{School of Mathematical Science, Yangzhou University\\
		Yangzhou 225002, China}
	\email{liurui20240048@163.com}
	
	\author[X. Q. Qiang]{Xiangqi Qiang}
	\address{School of Science, Jiangsu University of Science and Technology\\ Zhenjiang 212100, China}
	\email{xq.qiang@just.edu.cn}
	
	\author[C. J. Hou]{Chengjun Hou}
	\address{School of Mathematical Science, Yangzhou University\\
		Yangzhou 225002, China}
	\email{cjhou@yzu.edu.cn}
	

	\begin{abstract}
		
		We study diagonal bimodules of \'{e}tale groupoid $C^*$-algebras over their canonical diagonal subalgebras, and establish necessary and sufficient conditions for such a bimodule to be spectral-that is, determined by its spectrum. For a class of $\Gamma$-graded \'{e}tale groupoids, we prove  that the spectrality of diagonal bimodules is  equivalent to  their invariance under the action of the dual group $\widehat{\Gamma}$ in the abelian case, or under the coaction of $\Gamma$ in the nonabelian case, on the groupoid $C^*$-algebras,  both of which are induced by the underlying cocycle. This framework covers transformation groupoids arising from homeomorphism actions of countable groups, as well as from local homeomorphism actions of Ore semigroups. As applications, we characterize the spectrality of closed two-sided ideals and subalgebras that contain the diagonal subalgebra of \'{e}tale groupoid $C^*$-algebras.
	\end{abstract}
	
	
	\maketitle
	\section{Introduction}

	Many $C^*$-algebras, including the reduced crossed products of commutative $C^*$-algebras by discrete groups, graph $C^*$-algebras and higher-rank graph $C^*$-algebras, can be realized as the reduced $C^*$-algebras of \'{e}tale groupoids. Such a realization is often referred to as a coordinate description of the $C^*$-algebra and serves as an important tool for studying $C^*$-algebras, their subalgebras, ideals, and more general bimodules over canonical subalgebras. The aim of this paper is to study the bimodule structure of the reduced $C^*$-algebras of \'{e}tale groupoids over their diagonal subalgebras.
	
	The algebra $C_0(\mathcal{G}^{(0)})$ of all continuous functions vanishing at infinity on the unit space $\mathcal{G}^{(0)}$ of an \'{e}tale groupoid $\mathcal{G}$ embeds naturally into the reduced groupoid $C^*$-algebra $C_r^*(\mathcal{G})$, and is referred  as the diagonal subalgebra. A closed subspace $M$ of $C_r^*(\mathcal{G})$ is called a diagonal bimodule if $aMb\subseteq C_r^*(\mathcal{G})$ for all $a,b\in C_0(\mathcal{G}^{(0)})$. Clearly, each closed subalgebra of $C_r^*(\mathcal{G})$ containing the diagonal subalgebra, as well as each closed two-sided ideal of $C_r^*(\mathcal{G})$, is a diagonal bimodule. As every element in $C_r^*(\mathcal{G})$ can be viewed as a continuous function on $\mathcal{G}$ (\cite{Ren}), the spectrum $\sigma(M)$ of a diagonal bimodule $M$ is defined as the union of the open supports of all its elements. For an open subset $P\subseteq \mathcal{G}$, one can construct a diagonal bimodule $A(P)$. We refer to a bimodule of this form as a spectral bimodule. This naturally leads to the question of whether every diagonal bimodule is spectral.
	
	In \cite{MS}, Muhly and Solel proved that every diagonal bimodule of $C_r^*(\mathcal{G})$ is spectral when $\mathcal{G}$ is amenable, principal and ample. Later, Yan, Chen and Xu in (\cite{YCX}) extended this result by removing the ampleness assumption. This result, known as the spectral theorem for diagonal bimodules, serves as a fundamental tool in the study of (non-self-adjoint) subalgebras of groupoid $C^*$-algebras.

	However, when $\mathcal{G}$ is not principal, it is no longer true that every diagonal bimodule is spectral. For instance, if $\mathcal{G}$ is a nontrivial countable group, a diagonal bimodule is simply a closed subspace of the group algebra $C_r^*(\mathcal{G})$,
	and one can construct examples that are not spectral. In fact, the spectral theorem for diagonal bimodules fully characterizes the principality of groupoids: if every diagonal bimodule of $C_r^*(\mathcal{G})$ is determined by its spectrum, then $\mathcal{G}$
	must be principal.
	Related work includes \cite{HP}, which characterizes diagonal bimodules of the  $C^*$-algebra arising from the Cuntz groupoid, showing that they coincide with subspaces invariant under the gauge automorphisms and are generated by the Cuntz partial isometries they contain. In  \cite{HPP}, it is shown that for the groupoid $C^*$-algebra of a range-finite directed graph with no sources, a diagonal bimodule is spectral if and only if it is gauge-invariant. Later, this result is generalized to the higher rank graphes in \cite{Hop}.  Moreover, \cite[Theorem 5.7]{BEFPR} establishes the spectral theorem for a class of non-amenable principal groupoids defined by free actions of countable groups with certain approximation properties on compact spaces.

	The study of the spectral theorem for diagonal bimodules of principal \'{e}tale groupoid $C^*$-algebras revolves around two fundamental problems. The first is whether every element  $a\in C_r^*(\mathcal{G})$ can be approximated in the reduced norm by continuous functions compactly supported on its open support. The second concerns the characterization of elements in a diagonal bimodule whose open supports are bisections. We say that an \'{e}tale groupoid satisfies the Fourier coefficients approximation property if it fulfills the above approximation condition. This class of groupoids includes amenable groupoids, as well as transformation groupoids arising from homeomorphism actions of countable groups with certain approximation properties (\cite{BC,CN,FK}).
	
	In this paper, we characterize diagonal bimodules of groupoid $C^*$-algebras and establish necessary and sufficient conditions for such a bimodule to be spectral. In particular, we prove that for an \'{e}tale groupoid with the Fourier coefficients approximation property, a diagonal bimodule is spectral if and only if it is a closed linear span of all elements in it whose open supports are bisections. As a corollary, we  obtain that for a principal \'{e}tale groupoid satisfying the approximation property, each diagonal bimodule is spectral. The extends the applicability of the spectral theorem for bimodules of amenable principal groupoid $C^*$-algebras. We also prove that a two-sided ideal of $C_r^*(\mathcal{G})$ is spectral as a diagonal bimodule if and only if it is a dynamical ideal in the sense of  \cite{BCS}. Moreover, for topologically principal amenable \'{e}tale groupoids, we establish a one-to-one correspondence between open wide subgroupoids of $\mathcal{G}$ and $C^*$-subalgebras of $C_r^*(\mathcal{G})$  that contain $C_0^*(\mathcal{G}^{(0)})$ as a Cartan sualgebra, via the mapping $H\rightarrow A(H)$.
	
	When an \'{e}tale groupoid $\mathcal{G}$ admits a continuous cocycle $c$ into a countable group $\Gamma$, it is known that $c$ induces either an automorphism action of the dual group $\widehat{\Gamma}$ on  $C_r^*(\mathcal{G})$ when $\Gamma$ is abelian, or a coaction of $\Gamma$ on $C_r^*(\mathcal{G})$ when
	$\Gamma$  is nonabelian. In this setting, we show that a diagonal bimodule is spectral precisely when it is invariant under the canonical action of $\widehat{\Gamma}$  or the coaction of $\Gamma$, accordingly. Examples of such graded groupoids are also provided, including transformation groupoids arising from homeomorphism actions of countable groups, as well as from local homeomorphism actions of Ore semigroups. In particular, for transformation groupoid $C^*$-algebras associated to homeomorphism actions $\beta$ of  countable groups $\Gamma$, which  are either weakly amenable or satisfy the Haagerup-Kraus approximation property,  on locally compact Hausdorff spaces $X$, we show that closed two-sided ideals of the corresponding $C^*$-algebras are invariant under the actions of  dual group $\hat{\Gamma}$ or coactions of $\Gamma$ induced by the canonical $1$-cocycles  if and only if they are  generated by $C_0(U)$  for some $\beta$-invariant open subsets $U$ of $X$.
	
	The paper is organized as follows. In Section 2, we review preliminaries on \'{e}tale groupoids and groupoid $C^*$-algebras. Section 3 provides sufficient conditions under which diagonal bimodules of the reduced groupoid C*-algebras are determined by their spectra. In Section 4, we study the spectral theorem for diagonal bimodules in the setting of $\Gamma$-graded groupoids.

	\section{Preliminaries}

	In this paper, by an \'{e}tale groupoid we mean a second countable, locally compact, Hausdorff and topological groupoid. For the relevant notions and theories on \'{e}tale groupoids and their $C^*$-algebras, we refer to \cite{Ren, Sim}.
	
	For a groupoid $\mathcal{G}$, let $\mathcal{G}^{(0)}$ and $\mathcal{G}^{(2)}$ be its unit space and the set of composable pairs, respectively. The range and source maps from $\mathcal{G}$ to $\mathcal{G}^{(0)}$ will be denoted by $r$ and $s$, respectively. The inverse map from $\mathcal{G}$ onto itself is written by $\gamma\rightarrow \gamma^{-1}$ and the product map from $\mathcal{G}^{(2)}$ into $\mathcal{G}$ is written by $(\gamma,\gamma')\rightarrow \gamma\gamma'$. In the topological setting, we assume that $\mathcal{G}$ is a second countable, locally compact and Hausdorff topology space and that the structure maps are continuous, where $\mathcal{G}^{(2)}$ and $\mathcal{G}^{(0)}$ have the relative topology of $\mathcal{G}\times \mathcal{G}$ and $\mathcal{G}$, respectively. A topological groupoid is called to be \'{e}tale if its range and source maps are local homeomorphisms from $\mathcal{G}$ onto $\mathcal{G}^{(0)}$. A bisection is a subset $B$ of $\mathcal{G}$ such that both $r$ and $s$ restrict to injective maps on $B$. For $U\subset \mathcal{G}^{(0)}$, we let $\mathcal{G}_U=s^{-1}(U)$, $\mathcal{G}^U=r^{-1}(U)$ and $\mathcal{G}|_U=\mathcal{G}_U\cap \mathcal{G}^U$. If $\mathcal{G}$ is \'{e}tale, then $\mathcal{G}^{(0)}$ is a closed and open subset of $\mathcal{G}$, and $\mathcal{G}$ has a basis consisting of open bisections, and the range fibre $\mathcal{G}^u$ and the source fibre $\mathcal{G}_u$ over a unit $u\in \mathcal{G}^{(0)}$ are discrete in the relative topology. In particular, the isotropy group over  the unit $u$ given as $\mathcal{G}^u_u:=\mathcal{G}^u\cap \mathcal{G}_u$ is a discrete group. A subset $U$ of $\mathcal{G}^{(0)}$ is said to be invariant if $U=r(\mathcal{G}_U)$.
	
	Denote by $\mathcal{G}'=\cup_{u\in \mathcal{G}^{(0)}}\mathcal{G}_u^u$ the isotropy bundle of $\mathcal{G}$. An \'{e}tale groupoid  $\mathcal{G}$ is called \emph{principal} if $\mathcal{G}'=\mathcal{G}^{(0)}$, and is called \emph{topologically principal} if $\{u\in \mathcal{G}^{(0)}:\,\, \mathcal{G}^u_u=\{u\}\}$ is dense in $\mathcal{G}^{(0)}$. From \cite{Ren08}(or \cite[Lemma 3.3]{BCFC}), $\mathcal{G}$ is topologically principal if and only if the interior, $int(\mathcal{G}')$, of $\mathcal{G}'$ is equal to $\mathcal{G}^{(0)}$. Countable discrete groups, \'{e}tale equivalence relations on topological spaces, and transformation groupoids from dynamical systems all constitute important examples of \'{e}tale groupoids. A transformation groupoid arising from a homeomorphism action of a countable group on a  second countable, locally compact and Hausdorff space is principal (respectively, topologically principal) if and only if the action is free (respectively, topologically free).

	Given an \'{e}tale groupoid $\mathcal{G}$, the linear space, $C_c(\mathcal{G})$, of all continuous complex functions on $\mathcal{G}$ with compact supports is a $\ast$-algebra under the operations: $f^{\ast}(\gamma)=\overline{f(\gamma^{-1})}$ and $(f
	g)(\gamma)=\sum_{\gamma'\in \mathcal{G}_{s(\gamma)}}f(\gamma\gamma'^{-1})g(\gamma')$ for $f,g\in C_c(\mathcal{G})$ and $\gamma\in \mathcal{G}$.
	For each $u\in \mathcal{G}^{(0)}$, there is a $\ast$-representation $Ind_{u}$ of $C_c(\mathcal{G})$ on the Hilbert space $l^2(\mathcal{G}_u)$ of square summable functions on $\mathcal{G}_u$  by
	$Ind_{u}(f)(\xi)(\gamma)=\sum_{\gamma'\in \mathcal{G}_u}f(\gamma\gamma'^{-1})\xi(\gamma')$  for $f\in C_c(\mathcal{G})$, $\xi\in l^2(\mathcal{G}_u)$ and $\gamma\in \mathcal{G}_u$. The reduced $C^*$-algebra $C_r^*(\mathcal{G})$ of $\mathcal{G}$ is the completion of $C_c(\mathcal{G})$ with respect to the norm $\|f\|_{\operatorname{red}}=\sup_{u\in \mathcal{G}^{(0)}}\|Ind_{u}(f)\|$ for $f\in C_c(\mathcal{G})$.  Since $\mathcal{G}^{(0)}$ is clopen in $\mathcal{G}$, $C_c(\mathcal{G}^{(0)})$ is contained in $C_c(\mathcal{G})$ in the canonical way, and this extends to an injection $C_0(\mathcal{G}^{(0)})\hookrightarrow C_r^*(\mathcal{G})$. Thus  $C_0(\mathcal{G}^{(0)})$ is regarded as a $C^*$-subalgebra of  $C_r^*(\mathcal{G})$ and is referred as the diagonal subalgebra of $C_r^*(\mathcal{G})$.

	\section{diagonal bimodules of groupoid $C^*$-algebras}
	
	For an \'{e}tale groupoid $\mathcal{G}$, it follows from \cite{Ren,Sim} that there exists a norm-decreasing,  injective linear map $j$ from $C_r^*(\mathcal{G})$ into the abelian $C^*$-algebra $C_0(\mathcal{G})$ of all complex continuous functions on $\mathcal{G}$ vanishing at infinity,  such that
	$$j(ab)(\gamma)=\sum\limits_{\sigma\in \mathcal{G}_{s(\gamma)}}j(a)(\gamma\sigma^{-1})j(b)(\sigma) \mbox{ and } j(a^*)(\gamma)=\overline{j(a)(\gamma^{-1})}$$
	for all $a,b\in C_r^*(\mathcal{G})$ and $\gamma\in \mathcal{G}$.
	Moreover, $j(f)=f$ for $f\in C_c(\mathcal{G})$ or $f\in C_0(\mathcal{G}^{(0)})$.  For any complex function $f$ on $\mathcal{G}$, let $\overline{f}$ denote  its  complex conjugate, defined by $\overline{f}(\gamma)=\overline{f(\gamma)}$ for $\gamma\in \mathcal{G}$. Then the map $f \mapsto \overline{f}$ on $C_c(\mathcal{G})$ extends to   an isometric isomorphism  $a \mapsto  \overline{a}$ on $ C_r^*(\mathcal{G})$ such that $j(\overline{a})(\gamma)=\overline{j(a)(\gamma)}$ for all $a\in C_r^*(\mathcal{G})$ and $\gamma\in \mathcal{G}$ (\cite{YCX}). We call $\overline{a}$ the complex conjugate of $a$ in $C_r^*(\mathcal{G})$.
	By identifying each $a\in C_r^*(\mathcal{G})$ with $j(a)\in C_0(\mathcal{G})$, $C_r^*(\mathcal{G})$ is regarded as a subspace of $C_0(\mathcal{G})$.  For each $a\in C_r^*(\mathcal{G})$, define its open support as $\text{supp}'(a)=\{\gamma\in \mathcal{G}: a(\gamma)\neq 0\}$ and its support, $\text{supp}(a)$, as the closure of $\text{supp}'(a)$ in $\mathcal{G}$.
	
	As before, we denote by $ab$  the (convolution) multiplication in $C_c(\mathcal{G})$ and $C_r^*(\mathcal{G})$, and denote by $a\cdot b$  the point-wise multiplication in $C_0(\mathcal{G})$ , i.e., $(a\cdot b)(\gamma)=a(\gamma)b(\gamma)$. For a nonempty open subset $U$ of a locally compact, Hausdorff space $X$, we always assume that $C_c(U)$ is contained in $C_c(X)$ in the canonical way.

	\begin{remark} When $\mathcal{G}$ is a countably infinite  discrete group with   identity $e$, let $L(\mathcal{G})$ be the group von Neumann algebra generated by the regular representation $\lambda$ of $\mathcal{G}$ on $\ell^2(\mathcal{G})$,  and let $C_r^*(\mathcal{G})$ the reduced group $C^*$-algebra generated by $\lambda$. For $a\in L(\mathcal{G})$, define $\widehat{a}=a(\delta_e)\in \ell^2(\mathcal{G})$. Then $a\in L(\mathcal{G})\mapsto \widehat{a}\in \ell^2(\mathcal{G})$ is injective. In this case, we have $j(a)=\widehat{a}$ and $j(a)(g)=\langle a(\delta_e),\delta_g\rangle$ for all $a\in L(\mathcal{G})$ and $g\in G$, where $\{\delta_g:g\in \mathcal{G}\}$ is the canonical orthogonal basis of $\ell^2(\mathcal{G})$. \end{remark}

	\begin{definition} A norm closed subspace $M$ of $C_r^*(\mathcal{G})$ is called a bimodule over $C_0(\mathcal{G}^{(0)})$, if $ab,ba\in M$ for every $a\in M$ and $b\in C_0(\mathcal{G}^{(0)})$. For convenience, $M$ is referred to as a diagonal bimodule (of $C_r^*(\mathcal{G})$).
	\end{definition}

	\begin{definition} For a nonzero diagonal bimodule  $M$ of $C_r^*(\mathcal{G})$ and a nontrivial open subset $P$ of $\mathcal{G}$, we define
		$$\sigma(M):=\{\gamma\in \mathcal{G}:\,  a(\gamma)\neq 0\, \mbox{ for some $a\in M$}\},\quad  \hbox{ and }$$
		$$A(P):=\{a\in C_r^*(\mathcal{G}):\, a(\gamma)=0 \mbox{ for all $\gamma\notin P$}\}.$$
		For convenience, we let $\sigma(\{0\})=\emptyset$,  $A(\emptyset)=\{0\}$ and $A(\mathcal{G})=C_r^*(\mathcal{G})$.
	\end{definition}
	Clearly, each $A(P)$ is a diagonal bimodule of $C_r^*(\mathcal{G})$, and $\sigma(M)$ is an open subset of $\mathcal{G}$, called the \emph{spectrum}  for $M$. Also, $(M\cup C_c(\sigma(M))\subseteq A(\sigma(M))$ and $A(P)$ is closed under the complex conjugate. The following observation is a direct result of the definition.
	
	\begin{proposition}\label{p3} For each open subset $P$ of $\mathcal{G}$, we have
		$$P=\sigma(A(P)).$$
	\end{proposition}
	
	\begin{proof} Assume that $P$ is a nontrivial open subset of $\mathcal{G}$. For $\gamma\in P$, choose $f\in C_c(P)$ such that $f(\gamma)=1$. Thus $f\in A(P)$, so $\gamma\in \sigma(A(P))$. On the other hand, for $\gamma\in \sigma(A(P))$, choose $a\in A(P)$ such that $a(\gamma)\neq 0$. Then $\gamma\in P$, for otherwise, $a(\gamma)=0$, which is a  contradiction. Hence $P=\sigma(A(P))$.\end{proof}

	This proposition indicates that every open subset $P$ of $\mathcal{G}$ is uniquely characterized by the diagonal bimodule $A(P)$ associated with $P$, and the map, $P\mapsto A(P)$, from the family of open subsets of $\mathcal{G}$ into the diagonal bimodules of $C_r^*(\mathcal{G})$, is injective. This leads us to consider an analogous question for a diagonal bimodule $M$ of $C_r^*(\mathcal{G})$,  namely, whether $M$ can be similarly determined by its spectrum $\sigma(M)$, that is, whether the equality $M=A(\sigma(M))$ holds. In the case of amenable principal groupoids, the answer to this question is affirmative. The following theorem is drawn from \cite{MS,YCX}.
	
	\begin{theorem} {\rm (The spectral theorem for bimodules)} Let $\mathcal{G}$ be an amenable principal groupoid. Then the map $P\mapsto A(P)$ defines a one-to-one correspondence between open subsets of $\mathcal{G}$ and diagonal bimodules of $C_r^*(\mathcal{G})$.
	\end{theorem}

	However, the above theorem fails to hold when $\mathcal{G}$ is non-principal, as demonstrated by a counterexample in \cite{HP}. In the context of countable discrete groups, where a diagonal bimodule is simply a closed subspace of $C_r^*(\mathcal{G})$, one can construct a counterexample involving a proper diagonal bimodule $M$ whose spectrum is the entire group $\mathcal{G}$.
	
	\begin{example}\label{ex1}Let $G$ be a countably infinite discrete group.   For $g,h\in G$ with $g\neq h$, let $M=\mathbb{C}(\lambda_g+\lambda_h)$ be the one-dimensional subspace of $C_r^*(G)$. Then $\sigma(M)=\{g,h\}$, and $A(\sigma(M))=\mathbb{C}\lambda_g+\mathbb{C}\lambda_h$ is a two-dimensional subspace containing $M$ properly. 
	\end{example}
	
	\begin{example} Let $M:=\{a\in C_r^*(\mathbb{Z}): a(n)=a(-n) \mbox{ for all $n\in \mathbb{Z}$}\}$. Then it is the fixed point algebra of the action of $\mathbb{Z}$ on $C_r^*(\mathbb{Z})$  given by the automorphism,  $n\mapsto -n$ of $\mathbb{Z}$. Thus $M$ is a $C^*$-subalgebra of $C_r^*(\mathbb{Z})$ and $\sigma(M)=\mathbb{Z}$. Clearly, $A(\sigma(M))=C_r^*(\mathbb{Z})$ contains properly $M$.
	\end{example}

	The following proposition demonstrates that for an amenable \'{e}tale groupoid, principality is equivalent to the condition that every diagonal bimodule is determined by its spectrum. In \cite{HP}, a bimodule of the Cuntz groupoid $C^*$-algebra over $A(D_0)$ that is determined by its spectrum is referred to as reflexive. For convenience, we introduce the following definition.
	
	\begin{definition}
		A diagonal bimodule $M$ of $C_r^*(\mathcal{G})$  is said to be spectral if $M=A(\sigma(M))$.
	\end{definition}
	
	\begin{proposition} Let $\mathcal{G}$ be a non-principal  \'etale groupoid. Then there is a non-spectral diagonal bimodule of $C_r^*(\mathcal{G})$.
	\end{proposition}
	
	\begin{proof} By assumption, choose $\gamma_0\in \mathcal{G}'\setminus \mathcal{G}^{(0)}$ and define  $M : =\{a\in C_r^*(\mathcal{G}): a(\gamma_0)=a(r(\gamma_0))\}$. One can verify that $M$ is a diagonal bimodule of $C_r^*(\mathcal{G})$. Since $\gamma_0\neq r(\gamma_0)$, there exist $f_1, f_2\in C_c(\mathcal{G})$ such that $f_1(\gamma_0)=f_2(r(\gamma_0))=1$ and $f_1(r(\gamma_0))=f_2(\gamma_0)=0$. Let $f=f_1+f_2$. Then $f(\gamma_0)=f(r(\gamma_0))=1$, so $f\in M$, which implies that $\gamma_0, r(\gamma_0)\in \sigma(M)$. Choose $\psi\in C_c(\sigma(M))$ such that $\psi(\gamma_0)\neq \psi(r(\gamma_0))$. Then $\psi\in A(\sigma(M))$, but $\psi\notin M$.\end{proof}

	For a nonempty open subset $P$ of $\mathcal{G}$, define $A_c(P)=C_c(\mathcal{G})\cap A(P)$. Then $$C_c(P)\subseteq A_c(P)\subseteq A(P)\subseteq C_0(P).$$ By \cite[Lemma in Page 101]{Ren}, $C_c(P)$ is dense in $A_c(P)$ under the inductive limit topology on $C_c(\mathcal{G})$, and hence also  under the reduced norm on $C_c(\mathcal{G})$. In \cite{MS,YCX}, the authors used the bipolar theorem and the groupoid representation theory to show that $C_c(P)$ is dense in $A(P)$ with respect to the reduced norm for amenable  principal \'{e}tale groupoids. When $P$ is an open subgroupoid of an amenable \'{e}tale groupoid, this result remains valid by (\cite[Theorem 4.2]{BEFPR}). The following theorem demonstrates that the same conclusion holds for arbitrary open subsets of amenable \'{e}tale groupoids. Although this may be known, we include a proof for completeness, following the approach of \cite[Corollary 6.17, Chapter 5]{BO}.
	
	\begin{theorem}\label{th0} Let $\mathcal{G}$ be an amenable \'{e}tale groupoid. Then $C_c(P)$ is dense in $A(P)$ in the reduced norm on $C_r^*(\mathcal{G})$ for each nonempty  open subset $P$ of $\mathcal{G}$.
	\end{theorem}

	\begin{proof}
		Given a nonempty  open subset $P$ of $\mathcal{G}$, we first claim that $C_c(P)$ is dense in $A_c(P)$ in the reduced norm. 
		
		In fact, for $a\in A_c(P)$, we let $K:=\overline{\mathrm{supp}'(a)}\subseteq \cup_{i=1}^nU_i$ for open bisections $U_i's$. 
		By the partition of unity,   there exist  functions $h_i\in C_c(U_i)$ with $0\leq h_i\leq 1$ such that $a=\sum\limits_{i=1}^nh_i\cdot a$. Let $a_i:=h_i\cdot a\in A_c(P)$ with $\mathrm{supp}'(a_i)\subseteq P\cap U_i$. Thus, we may assume that the open support $\mathrm{supp}'(a)(:=V)$  of $a\in A_c(P)$ is an open bisection contained in $P$.

		For $n \geq 1$, define  $V_n:= \{\gamma\in \mathcal{G}: |a(\gamma)|>\frac 1n\}$ and $K_n:=\{\gamma\in \mathcal{G}: |a(\gamma)|\geq\frac 1n\}$. Then $K_n$ is compact, $V_n$ is open and $V_n\subseteq \overline{V_n}\subseteq K_n\subseteq V_{n+1}$ for each $n$, with   $\bigcup_{n=1}^{\infty}V_n=V$. 
		For each $n$, choose $\psi_n\in C_c(V_{n+1})$ such that $0\leq \psi_n(\gamma)\leq 1$ for $\gamma\in \mathcal{G}$ and $\psi_n|_{K_n}=1$. Let $f_n=\psi_n\cdot a$ for each $n$. Then $f_n\in C_c(V_{n+1}) \subseteq C_c(V)$, and $\{f_n\}$ converges to $a$ in the supremum norm on $C_c(\mathcal{G})$. Since   the reduced norm and the supremum norm coincide on the open bisection $V$ (\cite{Sim}), it follows that $\{f_n\}$ converges to $a$ in the reduced norm.  This proves the claim.

		Since $\mathcal{G}$ is amenable, there exists a sequence $\{\xi_n\}$ in $C_c(\mathcal{G})$ such that $\|\xi_n\|_{L^2(\mathcal{G})}\leq 1$ for each $n$ and $(\xi_n^* \xi_n)(\gamma)\rightarrow 1$  for $\gamma\in \mathcal{G}$, uniformly on compact subsets of $\mathcal{G}$.  
		Then, for each $n$, the multiplier map  $M_n:\, f\in C_c(\mathcal{G})\mapsto (\xi_n^*\xi_n)\cdot f\in C_c(\mathcal{G})$ extends to a contractive completely positive map, still denoted by $M_n$, on $C_r^*(\mathcal{G})$. Moreover, $M_n(a)\to a$ in $C_r^*(\mathcal{G})$ for every $a\in C_r^*(\mathcal{G})$ (See \cite[Corollary 6.17, Chapter 5]{BO}). Note that if $a\in A(P)$, then $M_n(a)=(\xi_n^*\xi_n)\cdot a\in A_c(P)$ for each $n$. Hence, $A_c(P)$ is dense in $A(P)$. Combining this with the previous claim, we conclude that $C_c(P)$ is dense in $A(P)$ in the reduced norm.
	\end{proof}

	As an application, the following corollary shows that for a group $G$, a closed subspace of $C_r^*(\mathcal{G})$ is spectral as a diagonal bimodule only if it is the closed linear span of some generators $\lambda_g's$.
	
	\begin{corollary}\label{l11} Let $G$ be an amenable countable infinite group. Then a nonzero closed subspace $M$ of $C_r^*(G)$ is spectral as a diagonal bimodule if and only if there is a unique subset $P\subseteq G$ such that $M=\overline{\operatorname{span}}\{\lambda_g:\, g\in P\}$.
	\end{corollary}
	\begin{proof} Assume that $M$ is spectral. Since $G$ is amenable, it follows from the above theorem that  $M=A(\sigma(M))=\overline{\operatorname{span}}\{\lambda_g:\, g\in \sigma(M)\}$.
		
		On the other direction, assume that $M=\overline{\operatorname{span}}\{\lambda_g:\, g\in P\}$ for a nontrivial subset $P$ of $G$. One can check that $\sigma(M)=P$. It follows from Theorem \ref{th0} that $A(\sigma(M))=A(P)=\overline{C_c(P)}=M$, thus $M$ is spectral.
	\end{proof}

	\begin{remark} When $\mathcal{G}$ is a countable discrete group, the sequence $\{j(a)(g)\}_{g\in \mathcal{G}}$ and the formal series $\sum\limits_{g\in \mathcal{G}}j(a)(g)\lambda_g$ are referred to as the Fourier coefficients and the Fourier series of $a\in C_r^*(\mathcal{G})$, respectively. The preceding theorem equivalently asserts  that for every amenable countable group, the Fourier series of each $a\in C_r^*(\mathcal{G})$ converges to $a$ in the reduced norm.  It is well known that this property also holds for certain non-amenable groups.
		
		In the more general setting of an \'{e}tale groupoid $C^*$-algebra $C_r^*(\mathcal{G})$, one may regard $j(a)(\gamma)$ as the $\gamma$-th Fourier coefficient of $a$.  In \cite{FK}, the authors studied the problem of approximating an element $a\in C_r^*(\mathcal{G})$ by continuous functions with compact support on $\mathrm{supp}'(a)$, with respect to the reduced norm on $C_r^*(\mathcal{G})$.
		For convenience, we introduce the following notion and its corollary.\end{remark}
	
	\begin{definition} An \'etale groupoid $\mathcal{G}$  is said to have the {\sl Fourier coefficients approximation property} if $a\in \overline{C_c(\mathrm{supp}'(a))}$ holds for all $a\in C_r^*(\mathcal{G})$, with the closure taken in the reduced norm.
	\end{definition}
	
	\begin{corollary}\label{c1}
		Assume an \'{e}tale groupoid $\mathcal{G}$ satisfies  the  Fourier coefficients approximation property. Then $C_c(P)$ is dense in $A(P)$ in the reduced norm on $C_r^*(\mathcal{G})$ for each nonempty  open subset $P$ of $\mathcal{G}$.
	\end{corollary}

	According to Theorem \ref{th0}, every amenable groupoid satisfies the Fourier coefficients approximation property. By results in \cite{BC,CN}, if $G$ is a countable discrete group that is weakly amenable or satisfies the Haagerup-Kraus approximation property, then every $C^*$-dynamical system defined by
	$G$  enjoys the Fej\'{e}r property. Consequently, the transformation groupoids arising from homeomorphism actions of $G$ on second countable, locally compact Hausdorff spaces also possess the Fourier coefficients approximation property. Furthermore, as shown in \cite{FK}, this approximation property holds for \'{e}tale groupoids that satisfy the rapid decay property with respect to conditionally negative definite lengths.
	
	We have the following observation.
	
	\begin{corollary}\label{c3} Let $\mathcal{G}$ be an \'{e}tale groupoid with the Fourier coefficients approximation property.
		\begin{enumerate}
			\item[\rm (i)] Let $\{P_n\}$ and $\{M_n\}$ be sequences of nonempty open subsets of $\mathcal{G}$ and of spectral diagonal bimodule of $C_r^*(\mathcal{G})$, respectively. Then $A(\cup_nP_n)=\overline{\operatorname{span}}\{A(P_n): n\geq 1\}$ and $\sigma(\overline{\operatorname{span}}\{M_n: n\geq 1\}))=\cup_n\sigma(M_n)$.
			
			\item[\rm (ii)] For nonempty open subsets, $\{P_i: 1\leq i\leq n\}$, of $\mathcal{G}$ and spectral diagonal bimodule $\{M_i: 1\leq i\leq n\}$, of $C_r^*(\mathcal{G})$, we have  $A(\cap_{i=1}^nP_i)=\cap_{i=1}^n A(P_i)$ and $\sigma(\cap_{i=1}^n M_i)=\cap_{i=1}^n\sigma(M_i)$.
		\end{enumerate}
	\end{corollary}

	For each open bisection $U$ of $G$, the two homeomorphisms $r|_U: U\rightarrow r(U)$ and $s|_U: U\rightarrow s(U)$ induce two $C^*$-isomorphisms, $r_{\ast}: C_0(U)\rightarrow C_0(r(U))$ and  $s_{\ast}: C_0(U)\rightarrow C_0(s(U))$ by $r_{\ast}(f)(u)=f((r|_U)^{-1}(u))$ and $s_{\ast}(f)(v)=f((s|_U)^{-1}(v))$ for  $f\in C_0(U)$, $u\in r(U)$ and $v\in s(U)$.

	\begin{lemma}\label{l1} Let $M$ be a diagonal bimodule and $U$ be an open bisection of $\mathcal{G}$.
		\begin{enumerate}
			\item[\rm (i)] If $f\in C_c(U)$ then $\|f\|_r=\|f\|_{\infty}$, so $C_0(U)\subset C_r^*(\mathcal{G})$ and $\|a\|_{r}=\|a\|_{\infty}$ for $a\in C_0(U)$.

			\item[\rm (ii)] If $f,g\in C_0(U)$ then $f\cdot g=r_{\ast}(f)g=fs_{\ast}(g)$. Moreover, if $f$ or $g$ is also in $M$, then $f\cdot g\in C_0(U)\cap M$.
			\item[\rm (iii)] 
			For $a\in M$ with $\mathrm{supp}'(a)\subseteq U$,   there exists a sequence $\{f_n\}$ in $M$ such that $f_n\in C_c(U)$ for each $n$ and $\{f_n\}$ converges to $a$ in the reduced norm.
			
		\end{enumerate}
	\end{lemma}
	
	\begin{proof} For (i), we refer to \cite{Sim}.
		
		For (ii), by calculation, we have $f(\gamma)g(\gamma)=(r_{\ast}(f)g)(\gamma)=(gs_{\ast}(f))(\gamma)$ for all $f,g\in C_c(U)$ and $\gamma\in \mathcal{G}$.
		
		For (iii),  let $V=\mathrm{supp}'(a)$. Then $V\subseteq U$ and $a\in C_0(U)$. Let $V_n, K_n, \psi_n$ and $f_n$ be as in the third paragraph of the proof for Theorem \ref{th0}. Then $f_n\in C_c(V_{n+1})\subseteq C_c(U)$. By (ii), we have $f_n\in M$ for each $n$ and  $\{f_n\}$ converges to $a$ in the reduced norm.
	\end{proof}
	
	\begin{proposition}\label{p2} Given an \'{e}tale groupoid $\mathcal{G}$  and a diagonal bimodule $M$ of $C_r^*(\mathcal{G})$, if the linear span of all elements in $M$ whose open supports are bisections is dense in $M$ in the reduced norm, then $M=\overline{C_c(\sigma(M))}$, where the closure is taken in the reduced norm.
	\end{proposition}
	
	\begin{proof} For $a\in M$,  we have $U:=\mathrm{supp}'(a)$ is contained in $\sigma(M)$. If $U$ is a bisection, then it follows from Lemma \ref{l1} (iii) that $a\in \overline{C_c(\sigma(M))}$. By the hypothesis, $M\subseteq \overline{C_c(\sigma(M))}$.
		
		On the other direction, let $K$ be a compact subset in $\mathcal{G}$ and $P$ be an open bisection such that $K\subset P\subset \sigma(M)$. For each $\gamma\in K$, there is $a\in M$ such that $a(\gamma)\neq 0$. From the assumption for $M$, we can let $Q:=\mathrm{supp}'(a)$ be a bisection. By the continuity of $a$, choose an open bisection neighbourhood $U_{\gamma}$ of $\gamma$  such that $U_{\gamma}\subset P\cap Q$ and $a(\tau)\neq 0$ for all $\tau\in U_{\gamma}$. For each $f\in C_c(U_{\gamma})$, let $\psi(\tau)=\frac{f(\tau)}{a(\tau)}$ for $\tau\in U_{\gamma}$. Then $\psi\in C_c(U_{\gamma})$, and $\psi\cdot a=f$ in $C_r^*(\mathcal{G})$. It follows from Lemma \ref{l1} (ii) that $f\in M$. This shows that $C_c(U_{\gamma})\subseteq M$. Since $K$ is compact, there exist open bisections $U_1,U_2,\cdots,U_n$ such that $K\subset \cup_{k=1}^nU_k\subset P$ and $C_c(U_k)$ is contained in $M$ for each $k$. Hence, by the partition of unity, if $b\in C_c(\sigma(M))$ has support contained in $K$, then $b\in M$. From the arbitrariness of $K$ and $P$, it follows from the partition of unity again that $C_c(\sigma(M))$ is contained in $M$.  \end{proof}

	\begin{theorem}\label{p1}  Let $\mathcal{G}$ be an \'{e}tale groupoid with the Fourier coefficients approximation property and  $M$ be a nonzero diagonal bimodule of $C_r^*(\mathcal{G})$. Then the following statements are equivalent.
		\begin{enumerate}
			\item[\rm (i)] $M$ is spectral;
			\item[\rm (ii)]  $M$ is generated linearly by all the elements whose open supports are bisections in $\mathcal{G}$;
			\item[\rm (iii)] for each $f\in C_c(\mathcal{G})$ and $a\in M$,  $f\cdot a\in M$.
		\end{enumerate}
	\end{theorem}
	\begin{proof}
		From Corollary \ref{c1}, $A(\sigma(M))=\overline{C_c(\sigma(M))}$. The equivalence of (i) and (ii) is derived from Proposition \ref{p2}.
		
		For (i) $\Rightarrow$ (iii), assume that $M$ is a spectral diagonal bimodule. Fix a nonempty open bisection $U$ of $\mathcal{G}$ and let $f\in C_c(U)$. For any $a\in M$, we choose a sequence $\{\varphi_n\}$ in $C_c(\sigma(M))$ converging  to $a$ in the reduced norm, thus in the uniform norm. Hence $\{f\cdot \varphi_n\}$ in $C_c(U)$ converges to $f\cdot a$ in the uniform norm of $C_0(\mathcal{G})$. From Lemma \ref{l1}, $\{f\cdot \varphi_n\}$ in $C_c(U)$ converges to $f\cdot a$ in the reduced norm of $C_r^*(\mathcal{G})$. Also since $f\cdot\varphi_n\in C_c(\sigma(M))\subseteq M$ for each $n$ and $M$ is closed, we have $f\cdot a\in M$. Using the partition of unity, we also have $f\cdot a\in M$ for each $f\in C_c(\mathcal{G})$ and $a\in M$.
		
		For (iii) $\Rightarrow$ (i), assume that $f\cdot a\in M$ for each $f\in C_c(\mathcal{G})$ and $a\in M$. In order to show that $M$ is spectral, we only show that $C_c(U)$ is contained in $M$ for any nonempty open bisection $U\subseteq \sigma(M)$.
		For such a bisection $U$, we consider $C_0(U)\cap M$.
		
		From Lemma \ref{l1}, $C_0(U)\cap M$ is a closed subspace in $C_0(U)$ in the uniform norm and is closed in the point-wise multiplicative. For $a\in C_0(U)\cap M$, let $V:=\text{supp}'(a)$. Then $V\subseteq U$. We adopt the notations, $V_n,K_n, \psi_n$,  in the proof of Theorem \ref{th0}. Let $\varphi_n=\psi_n\cdot \overline{a}$ for each $n$. Then $\varphi_n\in C_c(V_{n+1}) \subseteq C_c(U)$, and $\lim\limits_{n\rightarrow \infty}\varphi_n =\overline{a}$ in both of the supremum norm and the reduced norm (see Lemma \ref{l1}). Let
		$$g_n(\gamma) = \begin{cases}
			\frac{\varphi_n(\gamma)}{a(\gamma)}, & \text{for } \gamma \in V, \\
			0, & \text{for } \gamma \in \mathcal{G} \setminus \mathrm{supp}(\varphi_n).
		\end{cases}$$
		Then $g_n\in  C_c(V_{n+1})$ and $\varphi_n=g_n\cdot a$ for each $n$. By the assumption, we have $\varphi_n=g_n\cdot a\in M$ for each $n$. Thus $\overline{a}\in M$, so $C_0(U)\cap M$ is a $C^*$-subalgebra of $C_0(U)$.
		
		Given two distinct $\gamma_1,\gamma_2\in U$, it follows that there exist $a_1,a_2\in M$ such that $a_1(\gamma_1)\neq 0$ and $a_2(\gamma_2)\neq 0$. Let $b_1=\frac{1-a_2(\gamma_1)}{a_1(\gamma_1)}a_1+a_2$ and $b_2=a_1+\frac{2-a_1(\gamma_2)}{a_2(\gamma_2)}a_2$. Then $b_1,b_2\in M$, $b_1(\gamma_1)=1$ and $b_2(\gamma_2)=2$. Choose $f_1,f_2\in C_c(U)$ such that $f_i(\gamma_j)=\delta_{ij}$ for $i,j=1,2$. By the assumption, we have $f=f_1\cdot b_1+f_2\cdot b_2\in C_0(U)\cap M$ and $f(\gamma_1)=1$, $f(\gamma_2)=2$. The Stone-Werstrass theorem gives that $C_0(U)\cap M=C_0(U)$, thus $C_0(U)\subseteq M$.
	\end{proof}
	
	The following result extends the applicability of the spectral theorem for diagonal bimodules of principal groupoid $C^*$-algebras.
	
	\begin{corollary}  Let $\mathcal{G}$ be a principal  \'{e}tale groupoid with the Fourier coefficients approximation property. Then each diagonal bimodule $M$ of $C_r^*(\mathcal{G})$ is spectral.

	\end{corollary}
	\begin{proof} A result of Yan-Chen-Xu asserts that if $\mathcal{G}$ is a principal  \'{e}tale groupoid  and $M$ is a diagonal bimodule of $C_r^*(\mathcal{G})$,  then $f\cdot a\in M$ for each $f\in C_c(B)$ and $a\in M$, where $B$ is any open bisection of $\mathcal{G}$ (\cite[Proposition 2.1 (ii)]{YCX}). An argument by the partition of unity can imply that $f\cdot a\in M$ for each $f\in C_c(\mathcal{G})$ and $a\in M$, thus it follows from Theorem \ref{p1} that $M$ is spectral when $\mathcal{G}$ has the Fourier coefficients approximation property.\end{proof}
	
	\begin{remark}
		Let $\mathcal{G}$ be the transformation groupoid arising from a free homeomorphism action of a weakly amenable countable group $\Gamma$ (or a countable group $\Gamma$ with the  Haagerup-Kraus approximation property) on a second countable, locally compact and Hausdorff space $X$, then $\mathcal{G}$ is principal and satisfies the Fourier coefficients approximation property, thus the proceeding corollary shows that each diagonal bimodule $M$ of $C_r^*(\mathcal{G})$ is spectral. This result is consistent with Theorem 5.7 in \cite{BEFPR}.
		
	\end{remark}
	
	For diagonal bimodules $M$ and $N$ of $C_r^*(\mathcal{G})$, let $MN$ be the closure of linear span of $\{ab: a\in M, b\in N\}$, which is a diagonal bimodule. Remark that, for two open subsets $P, Q$ of $\mathcal{G}$, $PQ$ is also open in $\mathcal{G}$.
	
	\begin{corollary}\label{c2}
		Let $\mathcal{G}$ be an \'{e}tale groupoid with the Fourier coefficients approximation property.
		
		{\rm (i)} If $M$ and $N$ are two nonzero spectral diagonal bimodules of  $C_r^*(\mathcal{G})$, then $MN$ is also a spectral diagonal bimodule and $\sigma(MN)=\sigma(M)\sigma(N)$.
		
		{\rm (ii)} For two nonempty open subsets $P, Q$ of $\mathcal{G}$, we have  $A(PQ)=A(P)A(Q)$.
		
	\end{corollary}
	
	\begin{proof} (i) For $a\in M$ and $b\in N$ such that the open supports of $a$ and $b$ are bisections in $\mathcal{G}$, then $ab$ is in $MN$ and its open support is a bisection. From (ii) in  Theorem \ref{p1}, $MN$ is a spectral diagonal bimodule.
		
		If $\gamma\in \sigma(MN)$, then there exist $a\in M$ and $b\in N$ such that $0\neq ab(\gamma)=\sum\limits_{\gamma_1\gamma_2=\gamma}a(\gamma_1)b(\gamma_2)$. Hence there are $\gamma_1\in \sigma(M)$ and $\gamma_2\in \sigma(N)$ such that $\gamma=\gamma_1\gamma_2\in\sigma(M)\sigma(N)$.
		For the other direction, suppose that $\gamma_1\in \sigma(M)$ and $\gamma_2\in \sigma(N)$ such that $s(\gamma_1)=r(\gamma_2)$. Since $M$ and $N$ are spectral, it follows from (ii) in  Theorem \ref{p1} that there are $a\in M$ and $b\in N$ whose open supports are bisections in $\mathcal{G}$ such that $a(\gamma_1)\neq 0$ and $b(\gamma_2)\neq 0$. Hence $(ab)(\gamma_1\gamma_2)=a(\gamma_1)b(\gamma_2)\neq 0$, which implies that $\gamma_1\gamma_2\in\sigma(MN)$. Consequently, $\sigma(MN)=\sigma(M)\sigma(N)$.
		
		(ii) For two open subsets $P, Q$ of $\mathcal{G}$, since $A(P)$ and $A(Q)$ are spectral, it follows from (i) and Proposition \ref{p3} that $\sigma(A(P)A(Q))=\sigma(A(P))\sigma(A(Q))=PQ$, thus $A(P)A(Q)=A(PQ)$.
	\end{proof}
	
	Closed two-sided ideals and closed subalgebras containing $C_r^*(\mathcal{G}^{(0)})$ of $C_r^*(\mathcal{G})$ are two special types of diagonal bimodules. We have the following characterizations for them.
	
	\begin{lemma} {\rm (i)} Let $\mathcal{H}\subseteq \mathcal{G}$ be an open subgroupoid. Then $\mathcal{H}$ is invariant under the multiplication in the sense that $\mathcal{H}\mathcal{G}\subseteq \mathcal{H}$ and $\mathcal{G}\mathcal{H}\subseteq \mathcal{H}$ if and only there is an open invariant subset $U$ of $\mathcal{G}^{(0)}$ such that $\mathcal{H}=\mathcal{G}|_U$.
		
		{\rm (ii)} If $I$ is a closed two-sided ideal of $C_r^*(\mathcal{G})$, then $\sigma(I)$ is an open subgroupoid of $\mathcal{G}$ which is invariant under the multiplication.  Moreover, $\sigma(I)\cap \mathcal{G}^{(0)}=s(\sigma(I))$ is invariant in $\mathcal{G}^{(0)}$.
	\end{lemma}
	
	\begin{proof} (i) Assume that $\mathcal{H}\mathcal{G}\subseteq \mathcal{H}$ and $\mathcal{G}\mathcal{H}\subseteq \mathcal{H}$. Since $\mathcal{H}\subseteq \mathcal{G}$ is an open subgroupoid, $U:=s(\mathcal{H})=r(\mathcal{H})$ is open in $\mathcal{G}^{(0)}$. For $\gamma\in \mathcal{G}$ with $r(\gamma)\in U$, choose $\tau\in \mathcal{H}$ such that $s(\tau)=r(\gamma)$. Hence $\tau\gamma\in \mathcal{G}\mathcal{H}\subseteq \mathcal{H}$, and thus, $s(\gamma)=s(\tau\gamma)\in s(\mathcal{H})=U$. So, $U$ is invariant.
		
		On the other hand, let $\mathcal{H}=\mathcal{G}|_U$ for  an open invariant subset $U$ of $\mathcal{G}^{(0)}$. For $\tau\in \mathcal{G}$ and $\gamma\in \mathcal{H}$ with $s(\tau)=r(\gamma)$, we have $\tau\gamma\in \mathcal{G}\mathcal{H}$ and $s(\tau\gamma)=s(\gamma)\in U$. Since $U$ is invariant, we have $r(\tau\gamma)\in U$. Thus $\tau\gamma\in \mathcal{H}$, so $\mathcal{G}\mathcal{H}\subseteq \mathcal{H}$. Similarly, we have $\mathcal{H}\mathcal{G}\subseteq \mathcal{H}$.
		
		(ii) Let  $I$ be a closed two-sided ideal of $C_r^*(\mathcal{G})$.  For $\gamma\in \sigma(I)$, choose $a\in I$ such that $a(\gamma)\neq 0$. Remark that $I$ is self-adjoint. Thus $a^*\in I$, and $a^*(\gamma^{-1})=\overline{a(\gamma)}\neq 0$, so $\gamma^{-1}\in\sigma(I)$.
		
		For $\gamma\in\sigma(I)$ and $\tau\in \mathcal{G}$ with $s(\gamma)=r(\tau)$, choose $a\in I$ such that $a(\gamma)\neq 0$, and choose  an open bisection neighbourhood $B$ of $\tau$ and $f\in C_c(B)$ such that $f(\tau)=1$. Then $af\in I$ and $(af)(\gamma\tau)=a(\gamma)\neq 0$, so $\gamma\tau\in\sigma(I)$. Hence  $\sigma(I)\mathcal{G}\subseteq\sigma(I)$. By a similar argument, we have that $\mathcal{G}\sigma(I)\subseteq\sigma(I)$. Hence $\sigma(I)$ is invariant under the multiplication.
		
		It follows from (i) that  $\sigma(I)\cap \mathcal{G}^{(0)}=s(\sigma(I))$ is invariant in $\mathcal{G}^{(0)}$.\end{proof}
	
	Recall that a closed two-sided ideal $I$ of $C_r^*(\mathcal{G})$ is called a dynamical ideal if  there exists an open invariant subset $U$ of  $\mathcal{G}^{(0)}$ such that $I=C_r^*(\mathcal{G}|_U)$. The following is motivated by \cite{BCS}.
	
	\begin{theorem}\label{th2} Let $\mathcal{G}$ be an \'etale groupoid with the Fourier coefficients approximation property. If $I$ is a closed two-sided ideal of $C_r^*(\mathcal{G})$, then $A(\sigma(I))$ is the smallest dynamical ideal of $C_r^*(\mathcal{G})$ containing $I$. Moreover, $I$ is spectral as a diagonal bimodule if and only if $I$ is a dynamical ideal.
		
		Hence the map $\mathcal{H}\mapsto A(\mathcal{H})$ gives a one-to-one correspondence between open subgroupoids of $\mathcal{G}$ which are invariant under the multiplication, and the dynamical ideals of $C_r^*(\mathcal{G})$.
	\end{theorem}

	\begin{proof}  Let $I$ be a closed two-sided ideal of $C_r^*(\mathcal{G})$. From the proceeding lemma,  $U=s(\sigma(I))=r(\sigma(I))$ is invariant in $\mathcal{G}^{(0)}$ and $\sigma(I)=\mathcal{G}|_U$. Thus it follows from Theorem \ref{th0} that $A(\sigma(I))=\overline{C_c(\mathcal{G}|_U)}=C_r^*(\mathcal{G}|_U)$  is a dynamical ideal containing $I$. If $J$ is another dynamical ideal containing $I$, we let $J=C_r^*(\mathcal{G}|_V)$ for an open invariant subset $V$ of $\mathcal{G}^{(0)}$. Then $\mathcal{G}|_U=\sigma(I)\subseteq \sigma(J)=\mathcal{G}|_V$, so $A(\sigma(I))\subseteq J$. \end{proof}

	A subgroupoid $\mathcal{H}$ of an \'{e}tale groupoid $\mathcal{G}$ is called wide if $\mathcal{H}$ contains the unit space $\mathcal{G}^{(0)}$. The following is referred  to \cite[Theorem 3.3]{BEFPR}
	
	\begin{proposition}  Let $\mathcal{G}$ be a topologically principal, amenable and \'etale groupoid.  If $M$ is a $C^*$-subalgebra of $C_r^*(\mathcal{G})$ containing $C_0(\mathcal{G}^{(0)})$, then  $C_0(\mathcal{G}^{(0)})$ is a Cartan subalgebra of $M$ if and only if there exists a unique open wide subgroupoid $\mathcal{H}$ of $\mathcal{G}$ such that $M=C_r^*(\mathcal{H})$.
		In this case, $\mathcal{H}=\sigma(M)$, and $M$ is spectral as a diagonal bimodule.
		
		Consequently, the map $\mathcal{H}\mapsto A(\mathcal{H})$ gives a one-to-one correspondence between open wide subgroupoids of $\mathcal{G}$, and the $C^*$-subalgebra, containing $C_0^*(\mathcal{G}^{(0)})$ as a Cartan subalgebra, of $C_r^*(\mathcal{G})$.
	\end{proposition}
	\begin{proof} Assume that $M=C_r^*(\mathcal{H})$ for an open wide subgroupoid $\mathcal{H}$ of $\mathcal{G}$. Since each open wide subgroupoid of a topologically principal \'etale groupoid is also topologically principal, $C_0(\mathcal{G}^{(0)})$ is a Cartan subalgebra of $M$.
		
		On the other direction, assume that $C_0(\mathcal{G}^{(0)})$ is a Cartan subalgebra of $M$. Then the normalizer  $N(C_0(\mathcal{G}^{(0)}), M)$  of $C_0(\mathcal{G}^{(0)})$ in $M$ is dense in $M$. Remark that, if let $N(C_0(\mathcal{G}^{(0)}), C_r^*(\mathcal{G}))$  be the normalizer of $C_0(\mathcal{G}^{(0)})$ in $C_r^*(\mathcal{G})$, then $N(C_0(\mathcal{G}^{(0)}), M)=N(C_0(\mathcal{G}^{(0)}), C_r^*(\mathcal{G})) \cap M$.  Since $\mathcal{G}$ is  topologically principal, it follows from \cite[Proposition 4.8]{Ren08} that $N(C_0(\mathcal{G}^{(0)}), M)$ consists exactly of the elements of $M$ whose open support is an open bisection. From Theorem \ref{p1}, $M$ is spectral as a diagonal module, thus $M=\overline{C_c(\sigma(M))}$.
		
		We now claim that $\sigma(M)$ is a wide subgroupoid of $\mathcal{G}$. In fact, For $\gamma,\tau\in \sigma(M)$ satisfying $s(\gamma)=r(\tau)$, since the linear span of $N(C_0(\mathcal{G}^{(0)}), M)$ is dense in $M$, there exist $a,b\in N(C_0(\mathcal{G}^{(0)}), M)$ such that $a(\gamma)\neq 0$ and $b(\tau)\neq 0$. Noting that the open supports of $a$ and $b$ are bisections of $\mathcal{G}$, we have that
		$ab(\gamma\tau)=a(\gamma)b(\tau)\neq 0$, so $\gamma\tau\in\sigma(M)$. Clearly, $\sigma(M)$ contains $\mathcal{G}^{(0)}$ and is closed in the inverse, so it is an open wide subgroupoid of $\mathcal{G}$. Thus $M=C_r^*(\sigma(M))$.
	\end{proof}

	\section{ \'etale groupoids with  continuous $1$-cocycles }

	In this section, we let $\mathcal{G}$ be an \'etale groupoid satisfying  the Fourier coefficients approximation
	property and let $\Gamma$ be a countable  infinite  group with the identity $e$. If $c: \mathcal{G}\rightarrow \Gamma$ is a continuous $1$-cocycle, then $(\mathcal{G},c)$ is said to be  a $\Gamma$-graded groupoid. At this time, $\mathcal{G}_t:=c^{-1}(t)$ is an open-closed subset of $\mathcal{G}$ for each $t\in \Gamma$. In particular, $\mathcal{G}_e$ is an open-closed wide subgroupoid of $\mathcal{G}$.  By Corollary \ref{c1},   $A(\mathcal{G}_t)=\overline{C_c(\mathcal{G}_t)}$ for each $t\in \Gamma$. If $\Gamma$ and $\mathcal{G}_e$ are amenable, then it follows from \cite[Proposition 9.3]{Sp} (or \cite[Corollary 4.5]{RW}) that $\mathcal{G}$ is amenable.
	
	\subsection{Abelian groups} In this subsection, we assume that $\Gamma$ is abelian, and denote the group operation in $\Gamma$ as addition $+$, with the identity element as the zero element $0$. Let $\widehat{\Gamma}$ be the dual group of $\Gamma$. Recall that the cocycle $c$ gives rise to an automorphism action, denoted by $\alpha^c$, of $\widehat{\Gamma}$ on $C_r^*(\mathcal{G})$ characterized by
	$$\alpha^c_{\chi}(f)(\gamma)=\chi(c(\gamma)) f(\gamma)\quad \text{for $\chi\in \widehat{\Gamma}$, $f\in C_c(\mathcal{G})$ and $\gamma\in \mathcal{G}$}.$$
	By the continuity of the map $c$, the above equation holds for all $f\in C_r^*(\mathcal{G})$ \cite{Ren}. For $t\in \Gamma$ and $a\in C_r^*(\mathcal{G})$, define
	$$\Phi_t(a):=\int_{\widehat{\Gamma}} \overline{\chi(t)}\, \alpha^c_{\chi}(a) \,d\chi,$$
	and let
	$$C_r^*(\mathcal{G})_t:=\Phi_t(C_r^*(\mathcal{G}))$$
	be the range of $\Phi_t$. Then each $\Phi_t$ is a completely contractive and idempotent linear map on $C_r^*(\mathcal{G})$.

	In fact, for $t\in \Gamma$, $a\in C_r^*(\mathcal{G})$ and $\gamma\in \mathcal{G}$, we have
	$\Phi_t(a)(\gamma)= a(\gamma)$ when $c(\gamma)=t$, and $\Phi_t(a)(\gamma)=0$ when $c(\gamma)\neq t$. Hence, for each $t\in \Gamma$, $\Phi_t$ is exactly the bounded extension on $C_r^*(\mathcal{G})$ of the restriction map $f\in C_c(\mathcal{G})\mapsto f|_{\mathcal{G}_t}\in C_c(\mathcal{G}_t) (\subset C_c(\mathcal{G}))$. Thus, $C_r^*(\mathcal{G})_t=\overline{C_c(\mathcal{G}_t)}=A(\mathcal{G}_t)$.
	
	The following properties are from \cite{BFPR}.
	
	\begin{lemma}\label{l2.1}
		\begin{enumerate}
			\item[\rm (i)]  $C_r^*(\mathcal{G})_0$ is the fixed point algebra under the action $\alpha^c$, and $\Phi_0$ is a faithful conditional expectation from $C_r^*(\mathcal{G})$ onto $C_r^*(\mathcal{G})_0$.
			\item[\rm (ii)] $C_r^*(\mathcal{G})_t=\{a\in C_r^*(\mathcal{G}): \alpha^c_{\chi}(a)=\chi(t)a \mbox{ for all $\chi\in \widehat{\Gamma}$}\}$.
			\item[\rm (iii)] If $a\in C_r^*(\mathcal{G})_t$ and $b\in C_r^*(\mathcal{G})_s$, then $ab\in C_r^*(\mathcal{G})_{s+t}$.
			\item[\rm(iv)] If $a\in C_r^*(\mathcal{G})_t$ and $s\in \Gamma$, then $\Phi_s(a)=a$ when $s=t$, $\Phi_s(a)=0$ when $s\neq t$.
			\item[\rm (v)] If $a\in C_r^*(\mathcal{G})$, then $a\in \overline{\operatorname{span}}\{\Phi_t(a): t\in \Gamma\}$. Thus $C_r^*(\mathcal{G})=\overline{\operatorname{span}}\{C_r^*(\mathcal{G})_t: t\in \Gamma\}$.
			
		\end{enumerate}
	\end{lemma}
	
	\begin{definition}A diagonal bimodule $M$ of $C_r^*(\mathcal{G})$ is said to be invariant under the action $\alpha^c$ (or simply $\alpha^c$-invariant) if $\alpha^c_{\chi}(M)\subseteq M$ for all $\chi\in \widehat{\Gamma}$.
	\end{definition}
	From the definition of $\Phi_t$ and Lemma \ref{l2.1}(v), if a diagonal bimodule $M$ is $\alpha^c$-invariant then $\Phi_t(M)\subseteq M$ for each $t\in \Gamma$ and $M=\overline{\operatorname{span}}\{M\cap C_r^*(\mathcal{G})_t: t\in \Gamma\}$. From Lemma \ref{l2.1}(ii), $C_r^*(\mathcal{G})_t=A(\mathcal{G}_t)$ is an $\alpha^c$-invariant diagonal bimodule.

	\begin{proposition}\label{p4} Let  $(\mathcal{G},c)$ be a $\Gamma$-graded groupoid with $\Gamma$ abelian and $M$ be a diagonal bimodule of $C_r^*(\mathcal{G})$. Then the following statements hold.
		\begin{enumerate}
			\item[\rm (i)] If $M$ is $\alpha^c$-invariant, then $M$ is spectral if and only if $M\cap C_r^*(\mathcal{G})_t$ is spectral for each $t\in \Gamma$.
			\item[\rm (ii)] If $M$ is spectral, then $M$ is $\alpha^c$-invariant. Moreover, when $c^{-1}(0)=\mathcal{G}^{(0)}$, the converse also holds; that is, $M$ is spectral if and only if $M$ is $\alpha^c$-invariant.
		\end{enumerate}
	\end{proposition}
	
	\begin{proof}
		(i) Clearly, $\sigma(M\cap C_r^*(\mathcal{G})_t)\subseteq \sigma(M)\cap \mathcal{G}_t$ for each $t\in \Gamma$. Conversely, for $\gamma\in \sigma(M)\cap \mathcal{G}_t$, there exists $a\in M$ such that $a(\gamma)\neq 0$. Since $\Phi_t(a)(\gamma)=a(\gamma)$ and $\Phi_t(a)\in M\cap C_r^*(\mathcal{G})_t$, we have $\gamma\in \sigma(M\cap C_r^*(\mathcal{G})_t)$. Thus $\sigma(M\cap C_r^*(\mathcal{G})_t)=\sigma(M)\cap \mathcal{G}_t$. Consequently, it follows from Corollary \ref{c2} that $A(\sigma(M\cap C_r^*(\mathcal{G})_t))=A(\sigma(M)\cap \mathcal{G}_t)=A(\sigma(M))\cap C_r^*(\mathcal{G})_t$. Hence, from Lemma \ref{l2.1}(v), $M$ is spectral if and only if $M\cap C_r^*(\mathcal{G})_t$ is spectral for each $t\in \Gamma$.
		
		(ii) Assume that $M$ is spectral. Then $M=\overline{C_c(\sigma(M))}$. From the definition of $\alpha^c$, we have $\mathrm{supp}(\alpha^c_{\chi}(f))\subseteq \mathrm{supp}(f)$ for all $f\in C_c(\sigma(M))$ and $\chi\in \widehat{\Gamma}$, thus $C_c(\sigma(M))$ is $\alpha^c$-invariant, which implies that $M$ is $\alpha^c$-invariant.

		Now assume that $c^{-1}(0)=\mathcal{G}^{(0)}$. Then each $\mathcal{G}_t$ is a    clopen bisection of $\mathcal{G}$. If $M$ is $\alpha^c$-invariant, then, for each $t\in\Gamma$ and $a\in M\cap C_r^*(\mathcal{G})_t$, the open support of $a$ is contained in $\mathcal{G}_t$. From Theorem \ref{p1}, $M\cap C_r^*(\mathcal{G})_t$ is spectral, and thus, by (i), $M$ is spectral.
	\end{proof}
	\begin{remark} Assuming that $c^{-1}(0)\neq \mathcal{G}^{(0)}$, the $\alpha^c$-invariance of a diagonal bimodule $M$ does not necessarily imply the spectral property of $M$. Let $\Gamma$ be as before, $H$ be an amenable countable infinite group, and $\mathcal{G}:=\Gamma\times H$ be the direct product of $\Gamma$ and $H$. Let $c(t,h)=t$ be the projection to the first coordinate from $\mathcal{G}$ onto $\Gamma$. Then $c^{-1}(0)=\{0\}\times H$. For two distinct elements $h_1,h_2$ in $H$, let
		$M=\mathbb{C}(\lambda_{(0,h_1)}+\lambda_{(0,h_2)})$ be the one-dimensional subspace generated by the operator $\lambda_{(0,h_1)}+\lambda_{(0,h_2)}$ in $C_r^*(\mathcal{G})$, where $\lambda$ is the left regular representation of $\mathcal{G}$. One can verify that $\alpha^c_{\chi}(\lambda_{(0,h_1)}+\lambda_{(0,h_2)})=\lambda_{(0,h_1)}+\lambda_{(0,h_2)}$ for each $\chi\in\widehat{\Gamma}$, so $M$ is $\alpha^c$-invariant. However, $M$ is not spectral by Example \ref{ex1}.
		
	\end{remark}

	\begin{theorem}\label{Th1}
		Assume that the unit space $\mathcal{G}^{(0)}$ of $\mathcal{G}$ is totally disconnected, $\Gamma$ has a subsemigroup $\mathcal{P}$ containing the zero element $0$ and satisfying $\Gamma=\mathcal{P}-\mathcal{P}$, and the kernel $\mathcal{G}_0$ of $c$ is amenable and principal. If for each $p\in \mathcal{P}$, there exists a cover $\{U_{p,i} \mid i\in \Lambda_{p}\}$ of $\mathcal{G}_p$ by compact open bisections of $\mathcal{G}$ such that the collection $\{U_{p,i}U_{q,j}^{-1} \mid p,q\in \mathcal{P}, i\in \Lambda_{p}, j\in \Lambda_{q}, s(U_{p,i})=s(U_{q,j})\}$ forms an open cover of $\mathcal{G}$, then a diagonal bimodule $M$ of $C_{r}^{*}(\mathcal{G})$ is spectral if and only if $M$ is $\alpha^c$-invariant.
	\end{theorem}
	
	We firstly have the following observation.
	\begin{lemma}\label{l3.1}
		Let $\mathcal{F}_{0}=\{(p,q,i,j) \mid p,q\in \mathcal{P}, \ i\in \Lambda_{p}, \ j\in \Lambda_{q},\ s(U_{p,i})=s(U_{q,j})\}$. For $(p,q,i,j)\in \mathcal{F}_{0}$, let $W=s(U_{p,i})$.  Then the map $\psi: \mathcal{G}_{p-q}\cap \mathcal{G}^{r(U_{p,i})}_{r(U_{q,j})} \rightarrow \mathcal{G}_{0}\cap\mathcal{G}^{W}_{W}$  defined by $\gamma \mapsto U_{p,i}^{-1} \gamma U_{q,j}$ is a homeomorphism.
	\end{lemma}
	
	\begin{proof}
		For each $\gamma\in \mathcal{G}_{p-q}\cap \mathcal{G}^{r(U_{p,i})}_{r(U_{q,j})}$, we have that $c(\gamma)=p-q$ and there exist   elements $\tau\in U_{p,i}$ and $\sigma\in U_{q,j}$ such that $r(\gamma)=r(\tau)$ and $s(\gamma)=r(\sigma)$. Hence the product $\tau^{-1}\gamma\sigma$ is well-defined and lies in $\mathcal{G}_{0}\cap\mathcal{G}^{W}_{W}$.
		Now suppose $\tau_{1}, \tau_{2}\in U_{p,i}^{-1}\gamma U_{q,j}$. Then we can write $\tau_{1}=\sigma_{1}^{-1}\gamma\sigma_{1}'$ and $\tau_{2}=\sigma_{2}^{-1}\gamma\sigma_{2}'$ for some $\sigma_{1}, \sigma_{2}\in U_{p,i}$ and $\sigma_{1}', \sigma_{2}'\in U_{q,j}$. Since $r(\sigma_{1})=r(\gamma)=r(\sigma_{2})$ and $s(\sigma_{1}')=s(\gamma)=s(\sigma_{2}')$, it follows that $\sigma_{1}=\sigma_{2}$ and $\sigma_{1}'=\sigma_{2}'$. Therefore, $\tau_{1}=\tau_{2}$, which shows that $U_{p,i}^{-1}\gamma U_{q,j}$ is a single point, and thus, the map $\psi(\gamma)= U_{p,i}^{-1} \gamma U_{q,j}=\tau^{-1}\gamma\sigma$ is well-defined.
		Moreover, the map from $\mathcal{G}_{0}\cap\mathcal{G}^{W}_{W}$ onto $\mathcal{G}_{p-q}\cap \mathcal{G}^{r(U_{p,i})}_{r(U_{q,j})}$, defined by $\sigma \mapsto U_{p,i}\sigma U_{q,j}^{-1}$,  is the inverse of $\psi$. Therefore, $\psi$ is a homeomorphism.\end{proof}

	\begin{proof}[\bf{The proof of Theorem \ref{Th1}}]
		By the hypothesis, $\mathcal{G}$ is amenable. From Proposition \ref{p4}, it suffices to show that if $M$ is an $\alpha^c$-invariant nonzero diagonal bimodule of $C_{r}^{*}(\mathcal{G})$, then it is spectral.
		
		Let $P=\sigma(M)$. For each $(p,q,i,j)\in \mathcal{F}_{0}$, define $\Omega(p,q,i,j)=U_{p,i}U_{q,j}^{-1}$. Then the collection $\{\Omega(p,q,i,j) \mid (p,q,i,j)\in \mathcal{F}_{0}\}$ forms an open cover for $\mathcal{G}$, thus $P=\bigsqcup_{(p,q,i,j)\in \mathcal{F}_{0}}[P\cap \mathcal{G}_{p-q}\cap \mathcal{G}^{r(U_{p,i})}_{r(U_{q,j})}]$.
		In order to show that $M$ is spectral, by the partition of unity, it suffices to show that $C_{c}(P\cap \mathcal{G}_{p-q}\cap \mathcal{G}^{r(U_{p,i})}_{r(U_{q,j})})\subseteq M$ for every $(p,q,i,j)\in \mathcal{F}_{0}$.

		To prove this, given $(p,q,i,j)\in \mathcal{F}_{0}$ with $P(p,q,i,j)=P\cap \mathcal{G}_{p-q}\cap \mathcal{G}^{r(U_{p,i})}_{r(U_{q,j})}\neq \emptyset$, let $W=s(U_{p,i})=s(U_{q,j})$ and $Q=\psi(P(p,q,i,j))$, where $\psi$ is the homeomorphism from $\mathcal{G}_{p-q}\cap \mathcal{G}^{r(U_{p,i})}_{r(U_{q,j})}$ onto $\mathcal{G}_{0}\cap\mathcal{G}^{W}_{W}$ in Lemma \ref{l3.1}. Then $\psi$ induces a $*$-isomorphism $\pi: C_{c}(\mathcal{G}_{0}\cap\mathcal{G}_{W}^{W})\to C_{c}(\mathcal{G}_{p-q}\cap\mathcal{G}^{r(U_{p,i})}_{r(U_{q,j})})$ defined by $\pi(f)=f\circ\psi$. One can verify that $\pi(f)=\chi_{U_{p,i}}f\chi^{*}_{U_{q,j}}$ for $f\in C_{c}(\mathcal{G}_{0}\cap\mathcal{G}_{W}^{W})$. Thus $\pi$ extends to a bounded linear map, still denoted by $\pi$, from $A(\mathcal{G}_{0}\cap\mathcal{G}_{W}^{W})$ onto $A(\mathcal{G}_{p-q}\cap\mathcal{G}^{r(U_{p,i})}_{r(U_{q,j})})$. In particular, we have $\pi(C_{c}(Q))=C_{c}(P(p,q,i,j))$, and hence $\pi(A(Q))=A(P(p,q,i,j))$.

		Let $\theta_{p,i}: r(U_{p,i})\rightarrow s(U_{p,i})$ and $\theta_{q,j}: r(U_{q,j})\rightarrow s(U_{q,j})$ be the canonical homeomorphisms defined by the compact open bisections $U_{p,i}$ and $U_{q,j}$, respectively. For each $\phi\in C_0(\mathcal{G}^{(0)})$, we have $\phi_{p,i}:=\phi\circ\theta_{p,i}$ and $\phi_{q,j}:=\phi\circ\theta_{q,j}$ are in $C(r(U_{p,i}))$ and $C(r(U_{q,j}))$, respectively, and thus both are in $C_0(\mathcal{G}^{(0)})$. From the proof of Lemma \ref{l3.1}, we can verify that $\pi(\phi a)=\phi_{p,i}\pi(a)$ and $\pi(a \phi)=\pi(a)\phi_{q,j}$ for $\phi\in C_0(\mathcal{G}^{(0)})$ and $a\in A(\mathcal{G}_{0}\cap\mathcal{G}_{W}^{W})$.

		Now define $N:=\{a\in A(\mathcal{G}_{0}\cap\mathcal{G}_{W}^{W}) \mid \pi(a)\in M\cap A(P(p,q,i,j))\}$. Clearly, $N=\{a\in A(Q) \mid \pi(a)\in M\}$. Then by the above paragraph, $N$ is a diagonal bimodule in $C_{r}^{*}(\mathcal{G}_{0})$. Since $\mathcal{G}_{0}$ is amenable and principal, the Spectral Theorem for bimodules of Muhly-Solel implies that $N=A(\sigma(N))$. As $N\subseteq A(Q)$, it follows that $\sigma(N)\subseteq Q$. Conversely, let $\eta\in Q$, and choose $\gamma\in P(p,q,i,j)$ such that $\eta=\psi(\gamma)$. Note that $P(p,q,i,j)\subseteq P$. Choose $a\in M$ such that $a(\gamma)\neq 0$. Let $b=\chi_{U_{p,i}}^{*}\Phi_{p-q}(a)\chi_{U_{q,j}}$. Since $M$ is $\alpha^c$-invariant, we have $\Phi_{p-q}(a)\in M$ and $\mathrm{supp}'(b)\subseteq U_{p,i}^{-1}(c^{-1}(p-q)\cap \sigma(M))U_{q,j}\subseteq Q$. Thus $b\in A(Q)$ and $\pi(b)=\chi_{U_{p,i}}\chi_{U_{p,i}}^{*}\Phi_{p-q}(a)\chi_{U_{q,j}}\chi_{U_{q,j}}^*\in M$. So $b\in N$. By calculation, we have $b(\eta)=\Phi_{p-q}(a)(\gamma)=a(\gamma)\neq 0$. Hence $\eta\in \sigma(N)$, and thus $\sigma(N)=Q$. Consequently, $C_{c}(P(p,q,i,j))\subseteq \pi(A(Q))=\pi(A(\sigma(N)))=\pi(N)\subseteq M$.
	\end{proof}
	
	\subsection{Nonabelian groups} In this subsection, we replace dual actions with coactions to extend all results from the previous subsection to gradings by nonabelian groups.
	
	As before, let $(\mathcal{G},c)$ be a $\Gamma$-graded groupoid, where $\Gamma$ is nonabelian. Denote by $\lambda$ the regular representation of $\Gamma$ on $\ell^2(\Gamma)$, let $\{\delta_t\}_{t\in\Gamma}$ be the canonical orthonormal basis of $\ell^2(\Gamma)$, and let $\mathrm{Tr}$ be the canonical tracial state on $C_r^*(\Gamma)$, defined by $\mathrm{Tr}(a)=\langle a\delta_e,\delta_e\rangle$ for $a\in C_r^*(\Gamma)$. Let $\delta_{\Gamma}: C_{r}^{*}(\Gamma)\to C_{r}^{*}(\Gamma)\otimes C_{r}^{*}(\Gamma)$ be the comultiplication such that 
	$$\delta_{\Gamma}(\lambda_{t})=\lambda_{t}\otimes\lambda_{t} \quad \text{for } t\in \Gamma.$$
	By \cite[Lemma 6.1]{CRST}, the cocycle $c$ gives rise to a coaction $\delta_{c}$ of $\Gamma$ on $C_{r}^{*}(\mathcal{G})$ such that $\delta_{c}(f)=f\otimes\lambda_{t}$ for $f\in C_{c}(\mathcal{G})$ with $\mathrm{supp}(f)\subseteq c^{-1}(t)$ and $t\in \Gamma$. That is, $\delta_{c}: C_{r}^{*}(\mathcal{G})\to C_{r}^{*}(\mathcal{G})\otimes_{\min}C_{r}^{*}(\Gamma)$ is a nondegenerate $*$-homomorphism satisfying the coaction identity 
	$$(\delta_{c}\otimes \mathrm{id})\circ\delta_{c}=(\mathrm{id}\otimes\delta_{\Gamma})\circ\delta_{c}.$$ 
	The spectral subspaces of $C_{r}^{*}(\mathcal{G})$ for $\delta_c$ are defined as 
	$$C_{r}^{*}(\mathcal{G})_{t}:=\{a\in C_{r}^{*}(\mathcal{G}) : \delta_c(a)=a\otimes\lambda_{t}\} \quad \text{for } t\in \Gamma.$$ 
	Then the map $\Phi_{t}: C_{r}^{*}(\mathcal{G})\to C_{r}^{*}(\mathcal{G})_{t}$, defined by 
	$$\Phi_{t}(a)=(\mathrm{id}_{C_{r}^{*}(\mathcal{G})}\otimes \mathrm{Tr})\left(\delta_c(a)(1_{C_{r}^{*}(\mathcal{G})}\otimes\lambda_{t^{-1}})\right)$$
	for $a\in C_{r}^{*}(\mathcal{G})$, is a norm-decreasing linear map such that $\Phi_{t}(a)=a$ for $a\in C_{r}^{*}(\mathcal{G})_{t}$ and $\Phi_{t}(a)=0$ for $a\in C_{r}^{*}(\mathcal{G})_{s}$ with $s\neq t$. In particular, $\Phi_e$ is a conditional expectation from $C_r^*(\mathcal{G})$ onto $C_{r}^{*}(\mathcal{G})_{e}$.
	
	We list some properties of $\delta_c$ and the maps $\Phi_t$ as follows.
	
	\begin{lemma}\label{l2.3}
		\begin{enumerate}
			\item[\rm (i)] $C_r^*(\mathcal{G})_t=A(c^{-1}(t))$ for all $t\in\Gamma$.
			\item[\rm (ii)] $\Phi_t(a)(\gamma)=a(\gamma)$ for all $a\in C_r^*(\mathcal{G})$, $t\in\Gamma$ and $\gamma\in c^{-1}(t)$.
			\item[\rm (iii)] If $a\in C_r^*(\mathcal{G})_t$ and $b\in C_r^*(\mathcal{G})_s$, then $ab\in C_r^*(\mathcal{G})_{ts}$.
			\item[\rm (iv)] If $a\in C_r^*(\mathcal{G})$, then $a\in \overline{\operatorname{span}}\{\Phi_t(a): t\in \Gamma\}$. Thus $C_r^*(\mathcal{G})=\overline{\operatorname{span}}\{C_r^*(\mathcal{G})_t: t\in \Gamma\}$.
			
		\end{enumerate}
	\end{lemma}

	\begin{definition}
		A closed subspace $M$ of $C_{r}^{*}(\mathcal{G})$ is said to be invariant under the coaction $\delta_c$ (or simply $\delta_c$-invariant) if $\delta_c(M) \subseteq M \otimes_{\min} C_{r}^{*}(\Gamma)$, where $M \otimes_{\min} C_{r}^{*}(\Gamma) = \overline{\operatorname{span}}\{a \otimes b \mid a \in M, b \in C_{r}^{*}(\Gamma)\}$.
		
	\end{definition}

	\begin{lemma}
		If a closed subspace $M$ is $\delta_c$-invariant, then $\Phi_{t}(M) \subseteq M$ for all $t \in \Gamma$.\end{lemma}
	
	\begin{proof}
		By \cite{BFPR}, for each $t \in \Gamma$, there is a slice map $S_{t}: C_{r}^{*}(\mathcal{G}) \otimes_{\min} C_{r}^{*}(\Gamma) \to C_{r}^{*}(\mathcal{G})$ characterized by $S_{t}(a \otimes b) = \langle b(\delta_{e}), \delta_{t} \rangle a$ for $a \in C_{r}^{*}(\mathcal{G})$ and $b \in C_{r}^{*}(\Gamma)$ such that $\Phi_{t} = S_{t} \delta_c$.
		Note that $S_{t}(M \otimes C_{r}^{*}(\Gamma)) \subseteq M$. It follows from the invariance of $M$ that $\Phi_{t}(M) \subseteq M$ for each $t \in \Gamma$.
	\end{proof}

	\begin{proposition}\label{p5} Let  $M$ be a diagonal bimodule of $C_r^*(\mathcal{G})$. Then the following statements hold.
		\begin{enumerate}
			\item[\rm (i)] If $M$ is $\delta_c$-invariant, then $M$ is spectral if and only if $M\cap C_r^*(\mathcal{G})_t$ is spectral for each $t\in \Gamma$.
			\item[\rm (ii)] If $M$ is spectral, then $M$ is $\delta_c$-invariant. Moreover, when $c^{-1}(e)=\mathcal{G}^{(0)}$, the converse also holds; that is, $M$ is spectral if and only if $M$ is $\delta_c$-invariant.
		\end{enumerate}
	\end{proposition}
	
	\begin{proof}
		The proof is similar to that of Proposition 4.3.
	\end{proof}

	\begin{theorem}\label{th3}
		Assume that the unit space $\mathcal{G}^{(0)}$ of $\mathcal{G}$ is totally disconnected, $\Gamma$ has a subsemigroup $\mathcal{P}$ containing the identity $e$ and satisfying $\Gamma=\mathcal{P}\mathcal{P}^{-1}=\mathcal{P}^{-1}\mathcal{P}$, and the kernel $\mathcal{G}_e$ of $c$ is amenable and principal. If, for each $p\in \mathcal{P}$, there exists a cover $\{U_{p,i} \mid i\in \Lambda_{p}\}$ of $\mathcal{G}_p$ by compact open bisections of $\mathcal{G}$ such that the collection $\{U_{p,i}U_{q,j}^{-1} \mid p,q\in \mathcal{P}, i\in \Lambda_{p}, j\in \Lambda_{q}, s(U_{p,i})=s(U_{q,j})\}$ forms an open cover of $\mathcal{G}$, then a diagonal bimodule $M$ of $C_{r}^{*}(\mathcal{G})$ is spectral if and only if $M$ is $\delta_c$-invariant.
	\end{theorem}
	
	\begin{proof}By replacing $p-q$ with $pq^{-1}$ and $\mathcal{G}_0$ with $\mathcal{G}_e$, respectively, Lemma \ref{l3.1} still holds. One can prove this theorem using a similar argument to that of Theorem \ref{Th1}. \end{proof}
	
	\subsection{Examples} In this subsections, we provide some examples. Let $X$ be a second countable, locally compact and Hausdorff space, and let $\Gamma$ be an amenable countable discrete group.
	
	\begin{example}   A partial action $\beta$ of $\Gamma$ on $X$ consists of
		\begin{enumerate}
			\item[(1)] a collection $\{U_t\}_{t\in \Gamma}$ of open subsets $U_t$ of $X$,
			\item[(2)] a collection $\{\beta_t\}_{t\in \Gamma}$ of homeomorphisms $\beta_t: U_{t^{-1}}\to U_t$, $x\to tx$ such that $U_e=X$, $\beta_e$ is the identity map $\mathrm{id}_X$ on $X$; and for all $t_1,t_2\in \Gamma$, we have $t_2(U_{(t_1t_2)^{-1}}\cap U_{t_2^{-1}})=U_{t_2}\cap U_{t_1^{-1}}$, and $(t_1t_2)x=t_1(t_2x)$ for all $x\in  U_{(t_1t_2)^{-1}}\cap U_{t_2^{-1}}$.
		\end{enumerate}

		If $U_t=X$ for all $t\in \Gamma$, then $\beta$ is a group action of $\Gamma$ on $X$ by homeomorphisms.

		The transformation groupoid $\mathcal{G}$ attached to the partial action $\beta$ is given by
		$$\mathcal{G}:=\{(tx,t,x): t\in \Gamma, x\in U_{t^{-1}}\}$$
		with source map $s(tx,t,x)=x$, range map   $r(tx,t,x)=tx$,   composition $((t_1t_2)x,t_1,t_2x)(t_2x,t_2,x)=((t_1t_2)x,t_1t_2,x)$ and inverse $(tx,t,x)^{-1}=(x,t^{-1},tx)$. We equip $\mathcal{G}$ with the subspace topology from $X\times \Gamma\times X$. Then $\mathcal{G}$ is an amenable \'{e}tale groupoid whose unit space coincides with $X$. Moreover, there is a conical continuous $1$-cocycle $c: \mathcal{G}\to \Gamma$, $(tx,t,x)\mapsto t$, whose kernel $c^{-1}(e)$ is the unit space $X$.  By Proposition \ref{p4} and Proposition \ref{p5}, we have the following corollary.
		
		\begin{corollary}  Let $\beta$ be a $($partial$)$ action of $\Gamma$ on $X$ and $M$ be a diagonal bimodule of $C_r^*(\mathcal{G})$. Then $M$ is spectral if and only if $M$ is $\alpha^c$-invariant $($respectively, $\delta_c$-invariant$)$ when $\Gamma$ is abelian $($respectively, nonabelian$)$.
		\end{corollary}
	\end{example}

	\begin{example} Let $\mathcal{P}$ be a semigroup of $\Gamma$ that contains the identity $e$ and satisfies $\Gamma=\mathcal{P}\mathcal{P}^{-1}=\mathcal{P}^{-1}\mathcal{P}$. An injective (right) action $\theta$ of $\mathcal{P}$ on $X$ is defined as follows: for each $a\in \mathcal{P}$, the map $\theta_{a}: X\rightarrow X$ is a continuous, injective, open map. Moreover, the family $\{\theta_a\}_{a\in \mathcal{P}}$ satisfies the conditions $\theta_{a}\theta_{b}=\theta_{ba}$ for every $a,b \in \mathcal{P}$, and $\theta_e$ is the identity map $\mathrm{id}_X$ on $X$ (see \cite{QH}).

		Given such an injective  action, for $x\in X$, let $$Q_{x}:=\{g\in \Gamma: \exists a,b\in \mathcal{P},\, y\in X\; \text{such that} \; g=ab^{-1} \; \text{and} \; \theta_{a}(x)=\theta_{b}(y)\}.$$
		For  $x,y\in X$ and   $g\in \Gamma$, one can check that, $g\in Q_{x}$ if and only if there exists a unique element, denoted by $u(x,g)$, in $ X$ such that if $g=ab^{-1}$ for $a,b\in P$, then $\theta_{a}(x)=\theta_{b}(u(x,g))$
		Let $\mathcal{G}=\{(x,g)\in X\times \Gamma: g\in Q_x\}$ be the associated transformation groupoid. Recall that $\mathcal{G}$ is an amenable \'etale groupoid endowed with the operations: $(x,g)(y,h)=(x,gh)$ if $y=u(x,g)$, and $(x,g)^{-1}=(u(x,g),g^{-1})$, and with the relative product topology on $X\times \Gamma$. Let $c:\mathcal{G}\rightarrow \Gamma$ be the continuous $1$-cocycle defined by $c(x,g)=g$ for $(x,g)\in \mathcal{G}$. Then $c^{-1}(e)=\mathcal{G}^{(0)}$. By Proposition \ref{p4} and Proposition \ref{p5}, we have the following corollary.

		\begin{corollary} Let $\theta$ be an injective action of $\mathcal{P}$ on $X$ and $M$ be a diagonal bimodule of $C_r^*(\mathcal{G})$. Then $M$ is spectral if and only if $M$ is $\alpha^c$-invariant $($respectively, $\delta_c$-invariant$)$ when $\Gamma$ is abelian $($respectively, nonabelian$)$.
		\end{corollary}
	\end{example}

	\begin{remark} When $\Gamma$ is  weakly amenable or satisfies the Haagerup-Kraus approximation property, the transformation groupoid $\mathcal{G}$ associated to a homeomorphism action $\beta$ of $\Gamma$ on $X$  satisfies the Fourier coefficients approximation property, thus the preceeding corollary holds. Hence by Theorem \ref{th2}, a closed two-sided ideal of $C_r^*(G)$ is a dynamical ideal if and only if it is $\alpha^c$-invariant in the abelian case or $\delta^c$-invariant in the non-abelian case if and only if it is generated by $C_0(U)$ in $C_r^*(G)$ for some $\beta$-invariant open subset $U$ of $X$.

	\end{remark}
	
	\begin{example} Let $\mathcal{P}$ be a semigroup of $\Gamma$ that contains the identity $e$ and satisfies $\Gamma=\mathcal{P}\mathcal{P}^{-1}=\mathcal{P}^{-1}\mathcal{P}$. By a surjective local homeomorphism (right) action $\theta$ of $\mathcal{P}$ on $X$, we mean that for each $a\in \mathcal{P}$, the map $\theta_{a}: X\to X$ is a continuous, surjective local homeomorphism, $\theta_{a} \theta_{b}=\theta_{ba}$ for every $a,b \in \mathcal{P}$, and $\theta_e$ is the identity map $\operatorname{id}_X$ on $X$. Let $\mathcal{G}$ be the associated transformation groupoid \cite{RW}. Recall that 
		$$\mathcal{G}:=\{(x,m,y)\in X\times \Gamma\times X: \exists p,q\in \mathcal{P}, m=pq^{-1}, \theta_p(x)=\theta_q(y)\}.$$
		Then $r(x,m,y)=x$ and $s(x,m,y)=y$, where $(x,e,x)\in \mathcal{G}^{(0)}$ is identified with $x\in X$. For open subsets $U,V$ of $X$ and for $p,q\in \mathcal{P}$ such that $\theta_p: U\rightarrow \theta_p(U)$ and $\theta_q: V\rightarrow \theta_q(V)$ are homeomorphisms and $\theta_p(U)=\theta_q(V)$, let
		$$Z(U, p; q,V):=\{(x,pq^{-1},y): x\in U, y\in V, \theta_p(x)=\theta_q(y)\}.$$
		In particular, if we  let $Z(U,p):=Z(U,p; e,\theta_p(U))$, then $Z(U, p; q,V)=Z(U,p)Z(V,q)^{-1}$. Let $\mathfrak{F}$ be the family of all sets of the form $Z(U, p; q,V)$ with compact open subsets $U$ and $V$. Then $\mathfrak{F}$ forms a basis for the topology of $\mathcal{G}$, making $\mathcal{G}$ an amenable \'etale groupoid where each element in $\mathfrak{F}$ is a compact open bisection of $\mathcal{G}$. Let $c: \mathcal{G}\rightarrow \Gamma$ be the canonical continuous $1$-cocycle defined by $c(x,m,y)=m$ for $(x,m,y)\in \mathcal{G}$. Then $c^{-1}(e)=\{(x,e,y): \exists p\in \mathcal{P} \text{ such that } \theta_p(x)=\theta_p(y)\}$ is principal. As before, let $\alpha^c$ be the action of $\widehat{\Gamma}$ on $C_r^*(\mathcal{G})$ induced by $c$. From Theorem \ref{Th1}, Theorem \ref{th3}, and Theorem \ref{th0}, we have the following corollary.	
		
		\begin{corollary} Let $\theta$ be a surjective local homeomorphism action of $\mathcal{P}$ on a second countable, locally compact, Hausdorff, and totally disconnected space $X$. Let $\mathcal{G}$ be the associated transformation groupoid and $M$ be a diagonal bimodule of $C_r^*(\mathcal{G})$. Then the following statements are equivalent.
			\begin{enumerate}
				\item[\rm (i)] $M$ is spectral.
				\item[\rm (ii)]$M$ is generated as a diagonal bimodule by the partial isometries of the form $\chi_{Z(U,p;q,V)}$ that it contains.
				\item[\rm (iii)]$M$ is $\alpha^c$-invariant $($respectively, $\delta_c$-invariant$)$ when $\Gamma$ is abelian $($respectively, nonabelian$)$.

			\end{enumerate}
		\end{corollary}
	\end{example}
	
	\begin{remark} Let $T: \mathbb{N}^d \curvearrowright X$ be an action by $d$ commuting surjective local homeomorphisms, and $\mathcal{G}$ be the associated transformation groupoid. Let $\alpha$ be the canonical gauge action of $\mathbb{T}^d$ on $C_r^*(\mathcal{G})$. The proceeding corollary and Corollary \ref{c3} indicate that the map that carries an open subset $U$ of $\mathcal{G}$ to the diagonal bimodule  $A(P)$ is a lattice isomorphisms from the lattice of open subsets of $\mathcal{G}$ to the lattice of gauge-invariant diagonal bimodules of $C_r^*(\mathcal{G})$. This is a generalization of \cite[Proposition 3.9]{BCS} from Theorem \ref{th2}.
	\end{remark}


\begin{thebibliography}{999}
		
		
		\bibitem{BC} E. B\'{e}dos and R. Conti, {\sl Fourier series and twisted $C^*$-crossed products}, J. Fourier Anal. Appl., 21(2015), 32--75.
		
		
		\bibitem{BCFC} J. H. Brown, L. O. Clark, C. Farthing and A. Sims,  {\it Simplicity of algebras associated to \'{e}tale groupoids}, Semigroup Forum, 88(2014), 433--452.
		
		
		\bibitem{BCS} K. A. Brix, T. M. Carlsen and A. Sims, {\it Some results regarding the ideal structure of $C^*$-algebras of \'{e}tale groupoids}, J. London Math. Soc.,   109(2024), e12870.
		
		
		
		
		
		
		
		\bibitem{BEFPR} J. H. Brown, R. Exel, A. H. Fuller, D. R. Pitts  and S. A. Reznikoff, {\it Intermediate $C^*$-algebras of Cartan embeddings}, Proc. Amer. Math. Soc., Series B, 8(2021), 27--41.
		
		
		
		\bibitem{BFPR} J. H. Brown, A. H. Fuller, D. R. Pitts  and S. A. Reznikoff, {\it Graded $C^*$-algebras of and twisted groupoid $C^*$-algebras}, New York J. Math., 27(2021), 205--252.
		
		
		\bibitem{BO} N. P. Brown and N. Ozawa, {\it $C^*$-algebras and Finite-Dimensional Approximations}, Graduate Studies in Mathematics, Vol. 88,  American Mathematical Society Providence, Rhode Island, 2008.
		
		
		
		\bibitem{CN} J. Crann and M. Neufang, {\it A non-commutative Fej\'{e}r theorem for crossed products, the approximation property, and applications}, Int. Math. Res. Not. IMRN, 5(2022), 3571--3601.
		
		\bibitem{CRST} T. M. Carlsen,  E. Ruiz,  A. Sims  and M. Tomforde,  {\it Reconstruction of groupoids and $C^*$-rigidity of dynamical systems},  Adv. Math., 390(2021), Paper No. 107923, 55 pp.
		
		\bibitem{FK} A. H. Fuller and P. Karmakar, {\it Fourier coefficients and rapid decay in reduced groupoid $C^*$-algebras}, arXiv: 2412.05410, 2024.
		
		
		
		
		
		
		
		\bibitem{HP} A. Hopenwasser and J. R. Peters, {\it Subalgebras of the Cuntz $C^*$-Algebra}, arXiv: math/0304013, 2003.
		
		\bibitem{HPP} A. Hopenwasser, J. R. Peters and S. C. Power, {\it Subalgebras of graph $C^*$-algebras}, New York J. Math., 11(2005), 351--386.
		
		
		\bibitem{Hop} A. Hopenwasser, {\it The spectral theorem for bimodules in higher rank graph $C^*$-algebras}, Illinois J. Math., 49(2005), 993--1000.
		
		
		
		
		\bibitem{MS} P. S. Muhly and B. Solel, {\it Subalgebras of groupoid $C^*$-algebras}, J. Reine Angew. Math., 402(1989), 41--75.
		
		
		
		\bibitem{QH} X.   Qiang, C. Hou, {\it A note on the orbit equivalence of injective actions}, Filomat, 38(2024), 9067--9079.
		
		
		\bibitem{RW} J. Renault and D. Williams, {\it Amenability of groupoids arising from partial semigroup actions and topological higher rank graphs}, Trans. Amer. Math. Soc., 369(2017), 2255-2283.
		
		\bibitem{Ren} 	  J. Renault, {\it A groupoid approach to $C^*$-algebras},  Lecture Notes in
		Math., 793, New York, 1980.
		
		
		
		\bibitem{Ren08} J. Renault, {\it Cartan subalgebras in $C^*$-algebras}, Irish Math. Soc. Bull., 61(2008), 29--63.
		
		
		
		
		
		
		
		
		\bibitem{Sim} A. Sims,  {\it Hausdorff \'{e}tale groupoids and their $C^*$-algebras, in the volume:  Operator algebras and dynamics: groupoids, crossed products and Rokhlin dimension},   CRM Barcelona. Springer Nature Switzerland AG, Birkh\"{a}user, 2020.
		
		
		
		
		
		
		\bibitem{Sp} J. Spielberg, {\it Groupoids and $C^*$-algebras for categories of paths}, Trans. Amer. Math. Soc., 366(2014), 5771--5819.
		
		
		\bibitem{YCX} S. Yan, X. Chen, and S. Xu, {\it Submodules of the groupoid $C^*$-algebra},  Sci. China Ser. A., 41(1998), 379--385
		
		
		
		
	\end{thebibliography}
\end{document}